\documentclass[11 pt]{article}

\usepackage[latin1]{inputenc} 
\usepackage[T1]{fontenc}     
\usepackage[english]{babel}

\usepackage{cite}

\usepackage[top=3cm, bottom=3cm, left=4cm, right=4cm]{geometry}		
\usepackage{chngpage}

\usepackage{color}

\usepackage{graphicx}

\usepackage{tikz}

\usepackage[francais]{minitoc}

\usepackage{amsthm}											
\usepackage{amsmath}												
\usepackage{amssymb}
\usepackage{mathrsfs}

\newtheorem{definitionEng}{Definition}[section]
\newtheorem{theorem}{Theorem}[section]
\newtheorem{corollary}{Corollary}[section]
\newtheorem{lemma}{Lemma}[section]
\newtheorem{remark}{Remark}[section]
\newtheorem{remarks}{Remarks}[section]
\newtheorem{example}{Example}[section]

\usepackage{fancyhdr}
\newcommand\R{\mathbb{R}}
\newcommand\N{\mathbb{N}}

\newcommand{\D}{\, \mbox{d}}

\newcommand{\norm}[1]{\left\Vert #1\right\Vert}							

\begin{document}

\title{\bf Convergence Rate of Nonlinear Switched Systems}
\author{Philippe \textsc{JOUAN} and Sa\"id \textsc{NACIRI}\footnote{LMRS, CNRS UMR 6085, Universit\'e
    de Rouen, avenue de l'universit\'e BP 12, 76801
    Saint-Etienne-du-Rouvray France. E-mail: Said.Naciri1@univ-rouen.fr and Philippe.Jouan@univ-rouen.fr}}

\date{\today}

\maketitle

\begin{abstract}
 This paper is concerned with the convergence rate of the solutions of nonlinear switched systems.

We first consider a switched system which is asymptotically stable for a class of inputs but not for all inputs. We show that solutions corresponding to that class of inputs converge arbitrarily slowly
to the origin.

Then we consider analytic switched systems for which a common weak quadratic 
Lyapunov function exists. Under two different sets of assumptions we provide explicit exponential convergence rates for inputs with a fixed dwell-time.

\vskip 0.8cm


\vskip 0.2cm


\end{abstract}

\vskip 1cm


\section{Introduction}

Switched systems combine continuous dynamics with discrete events \cite{L03}, and have attracted growing interest in recent years. In particular, problems related to the stability of switched systems have received a lot of attention \cite{L03,LM99}. Among these issues the convergence rate of solutions is of capital importance in practical situations. Indeed, knowing that the solutions converge to the origin is not enough if this information does not come with an estimate of the rate of decay of the solutions. 

A switched system is a parametrized family of vector fields together with a law selecting at each time which vector field is responsible for the evolution of the system. This law is in general a piecewise constant and right-continuous function of time and is called an input or switching signal. For these inputs a solution of the switched system is merely a concatenation of integral curves of the different vector fields. 
It is a well-known fact that switching among asymptotically stable vector fields can lead to trajectories that escape to infinity \cite{LM99}. On the other hand switching adequately between unstable vector fields may lead to trajectories that converge to the origin. In \cite{LM99} Liberzon and Morse identify three basic problems related to the stability of switched systems. These problems are:
\begin{enumerate}
	\item Find conditions guaranteeing that the switched system is asymptotically stable for all inputs;
	\item Identify the largest class of switching signals for which the switched system is asymptotically stable;
	\item If the vector fields are not asymptotically stable, find at least one switching signal that makes the switched system asymptotically stable.
\end{enumerate}
In the papers \cite{BJ11} and \cite{JN13} (see also \cite{BM05,SVR11}) systems asymptotically stable for some class of inputs but not for all inputs are identified. For these systems two questions naturally arise. Does it exist a common convergence rate for all inputs of the class under consideration? If the answer is negative is it possible to determine some convergence rate depending of a parameter, for instance a dwell-time?

The paper is an attempt to answer these questions. It is organized as follows: after having recalled some basic facts in Section \ref{SectionBasicdef} we show in Section \ref{SectionNoconvergence} that if the asymptotic stability restricts to some class of inputs, for instance dwell-time inputs or nonchaotic ones \cite{BJ11,JN13}, then the solutions corresponding to these switching signals converge arbitrarily slowly to the origin.

Then we turn our attention to switched systems defined by a set of analytic vector fields that share a common weak Lyapunov function (this is the framework of the previous papers \cite{BJ11,BJN14,JN13,RSD10,SVR11}). In Section \ref{SectionHomogeneous} we consider homogeneous vector fields and quadratic weak Lyapunov functions and we prove exponential stability, with a computable convergence rate, for all inputs with a fixed dwell-time. In Section \ref{SectionGeneral} this result is extended to general non linear switched systems under the assumptions that the differentials of the vector fields at the origin are Hurwitz and that the Hessian of the Lyapunov function is positive definite. Some partial results about linear systems can be found in \cite{JN15}.

\section{Basic definitions}\label{SectionBasicdef}

\subsection{The switched system}

Let $\{f_1,\ldots,f_p\}$ be a finite family of smooth vector fields of $\mathbb{R}^d$. Here and subsequently we assume the origin to be a common
equilibrium of the vector fields.\\
Consider the switched system
	$$\dot{x}=f_{u}(x),\quad x\in\mathbb{R}^d,\quad u\in\{1,\ldots,p\}.$$
An input, or switching law, is a piecewise constant and right-continuous function $t\longmapsto u(t)$ from $[0,+\infty)$ to $\{1,\ldots,p\}$.
Fixing such a switching signal gives rise to a non autonomous differential equation
	$$\dot{x}=f_{u(t)}(x),\quad x\in\mathbb{R}^d.$$
Let us denote by $\Phi_u^t(x)$, $t\geq0$, the solution of the switched system for the input $u$ and for the initial condition $x\in\R^d$.
Let also $\Phi_i^t(x)$ stand for the flow of $f_i$ which corresponds to the case where the input is constant and equal to $i$.
The inputs being piecewise constant a solution of the switched system is nothing but a concatenation of integral curves of the vector fields $f_i$, $i=1,\ldots,p$.
In the literature, one may encounter inputs that are assumed to be merely measurable. In this context, a solution is an absolutely continuous curve that satisfies the differential equation almost everywhere. 

\subsection{Stability of switched systems}

Here we introduce the concept of stability we are interested in and the notion of convergence rate. 

\begin{definitionEng}
A continuous function $\alpha:[0,+\infty) \longrightarrow \mathbb{R}_+$ is said to be of class $\mathcal{K}$ (resp. $\mathcal{K}_\infty$) if 
$\alpha(0)=0$ and it is increasing (resp. $\alpha(0)=0$ and it is increasing to $+\infty$).

A continuous function $\beta:[0,+\infty)\times [0,+\infty)\longrightarrow \mathbb{R}_+$ is said to be of class $\mathcal{KL}$ if for each 
$t\in\mathbb{R}_+$, the function $\beta(\cdot,t)$ is of class $\mathcal{K}_\infty$ and for each $r\in\mathbb{R}_+$ the function $\beta(r,\cdot)$
decreases to $0$.
\end{definitionEng}

\begin{definitionEng}
The switched system is said to be GUAS (Globally Uniformly Asymptotically Stable) if there exists a class $\mathcal{KL}$ function $\beta$ such that
for all inputs $u$ and all initial conditions $x\in\mathbb{R}^d$, 
	$$\forall t\geq0, \quad \|\Phi_u^t(x)\| \leq \beta(\|x\|,t).$$ 
The function $\beta$ provides a \emph{convergence rate} of the solutions to the equilibrium.\\
In particular, if there exist two positive constants $C$ and $\mu$ such that $\beta(r,t)=Cre^{-\mu t}$ then the \emph{convergence rate} is said to be exponential.	
\end{definitionEng}

It is well-known that as soon as an homogeneous system is GUAS the convergence rate is exponential \cite{H04}. Moreover, the GUAS property of homogeneous switched systems is equivalent to the local attractivity of the origin \textit{i.e.} all solutions near the origin converge to zero.  

An equivalent definition of GUAS switched systems can be given in terms of $\varepsilon$ and $\delta$. 

\begin{definitionEng}
The switched system is said to be GUAS if 
	$$\forall \varepsilon, \delta>0, \exists T>0	\quad \text{s.t.} \quad \|x\|\leq\delta \Longrightarrow \forall t\geq T, \quad \|\Phi_u^t(x)\|\leq \varepsilon.$$
\end{definitionEng}

\subsection{Classes of inputs}

In this section, the main classes of inputs that can be found in the literature \cite{H04,BJ11} are introduced. 

\begin{definitionEng}
An input is said to have a \emph{dwell-time} $\delta>0$ if the time between two discontinuities is at least $\delta$.
\end{definitionEng}

An input that presents a finite number of discontinuities has a dwell-time.

\begin{definitionEng}
An input $u$ is said to have an \emph{average dwell-time} if there exist two positive numbers $\delta$ and $N_0$ such that for any $t\geq 0$ the number $N(t,t+\tau)$ of discontinuities of $u$
on the interval $]t,t+\tau[$ satisfies 
		$$N(t,t+\tau)\leq N_0+\dfrac{\tau}{\delta}.$$
The number $\delta$	is called the \emph{average dwell-time} and $N_0$ the \emph{chatter bound}.	
\end{definitionEng}

\begin{definitionEng}
An input $u$ is said to have a \emph{persistent dwell-time} if there exists a sequence of times $(t_k)_{k\geq0}$ increasing to $+\infty$ and two positive numbers $\delta$
and $T$ such that $u$ is constant on $[t_k,t_k+\delta[$ and $t_{k+1}-t_k\leq T$ for all $k\geq0$. 
The constant $\delta$ is called the \emph{persistent dwell-time} and $T$ the \emph{period of persistence}.
\end{definitionEng}

\begin{definitionEng}[See \cite{BJ11}]
An input $u$ is said to be chaotic if there exists a positive constant $\tau$ and a sequence $([t_k,t_k+\tau])_{k\geq0}$ of intervals that satisfies the following conditions:
\begin{enumerate}
	\item $t_k \longrightarrow_{k \to +\infty} +\infty$;
	\item For all $\varepsilon>0$ there exists $k_0\in\N$ such that for all $k\geq k_0$, the input $u$ is
		constant on no subinterval of $[t_k , t_k + \tau ]$ of length greater than or equal to $\varepsilon$.
\end{enumerate}
An input that does not satisfy these conditions is called a nonchaotic input. 
\end{definitionEng}

\begin{definitionEng}[See \cite{BJ11}]
An input $u$ is said to be regular if it is nonchaotic and for each $i\in\{1,\ldots,p\}$ there exist infinitely many disjoint intervals of the same length on which $u$ is equal to $i$.

\end{definitionEng}

\vskip 0,4cm

\noindent It is easy to show the following inclusions between the different classes of inputs: 
\begin{enumerate}
	\item An input that has a dwell-time has an average dwell-time;
	\item An input that has an average dwell-time is nonchaotic;
	\item A nonchaotic input has a persistent dwell-time.
\end{enumerate} 
Notice that all these inclusions are strict.


\section{Systems with arbitrarily slow convergence}\label{SectionNoconvergence}

We are now in a position to state the first main result of this paper. It basically says that if the switched system is globally asymptotically stable for one class of inputs 
 but not GUAS then the corresponding solutions converge arbitrarily slowly to the origin.
 
  This result also shows that the GUAS property is equivalent to the existence
 of a convergence rate for at least one of the classes of inputs defined above. This is moreover the way the theorem is proved.\\

\begin{theorem}
Let $\mathcal{U}$ be one of the classes of inputs listed above.\\
If the switched system is globally asymptotically stable over the class $\mathcal{U}$ but not for all inputs, then there is no
uniform convergence rate for the class $\mathcal{U}$. 
\end{theorem}

\begin{proof}
Let us assume that there exists a class $\mathcal{KL}$ function $\beta$ such that for all $\overline{u}\in\mathcal{U}$ holds: 
	$$\forall x\in\mathbb{R}^d,\ \ \forall t\geq0, \quad \|\Phi_{\overline{u}}^t(x)\|\leq \beta(\|x\|,t).$$
Let $u$ be any piecewise constant input. Let us fix a positive number $T>0$ and let us consider the input $\overline{u}$ defined by:
$$
\overline{u}(t)=\left\{
\begin{array}{ll}
u(t) & \mbox{ if } t\in[0,T)\\
1 & \mbox{ if } t\geq T
\end{array}
\right.$$
 This input has a finite number of switches, hence a dwell-time, so that $\overline{u}\in\mathcal{U}$, and one has for all $x$, 
	$$\forall t\geq0, \quad \|\Phi_{\overline{u}}^t(x)\|\leq \beta(\|x\|,t).$$
It follows that for each $t\in[0,T)$, $\|\Phi_u^t(x)\|\leq \beta(\|x\|,t)$. Since $T$ can be any positive number, we deduce:  
 	$$\forall t\geq0, \quad \|\Phi_u^t(x)\| \leq \beta(t,\|x\|).$$
Since $u$ is arbitrary the switched system is GUAS, a  contradiction.

\end{proof}

\begin{corollary}
If the switched system is GAS over one of the classes previously defined and admits a convergence rate $\beta$ over this class then it is GUAS and $\beta$ is a convergence rate for all measurable inputs.
\end{corollary}
\begin{proof}
Only the last assertion should be proven. Actually on a compact interval $[0,T]$, any measurable input is a weak-$\ast$ limit of a sequence of piecewise constant inputs, the corresponding solutions from the same initial point are uniformly convergent, and $\beta$ is a convergence rate on that interval. Since $T$ is arbitrary the proof is finished.

\end{proof}

\begin{remark}
The same result holds for the convexified system:
$$
\dot{x}=\sum_{i=1}^{p}v_if_i(x),\quad x\in\mathbb{R}^d,
$$
where $v=(v_1,\dots,v_p)$ satisfies $\sum_{i=1}^{p}v_i=1$ and $v_i\geq 0$ for $i=1,\dots,p$. Indeed the measurable inputs of this system are also  weak-$\ast$ limits of sequences of piecewise constant inputs of the original system, so that the assumption of the corollary also implies that the convexified system is GUAS and that $\beta$ is a convergence rate for all measurable inputs of this system.
\end{remark}


\section{Convergence rate of homogeneous switched systems}\label{SectionHomogeneous}

In this section, we deal with switched systems defined by a finite set $\{f_1,\ldots,f_p\}$ of analytic homogeneous vector fields, that is analytic vector fields that moreover satisfy:
$$
\forall x\in \mathbb{R}^d,\ \ \forall \lambda \in \mathbb{R} \qquad f_i(\lambda x)=\lambda f_i(x).
$$
Clearly this setting includes the linear case.

We assume the vector fields $\{f_1,\ldots,f_p\}$ to be globally asymptotically stable and to share a common, but not strict in general, quadratic Lyapunov function.
More accurately there exists a symmetric positive definite matrix $P$ such that the positive definite function defined on $\R^d$ by $V(x)=x^TPx$ is a weak Lyapunov function for each vector field \textit{i.e.}
		$$\forall i\in\{1,\ldots,p\},\ \  \forall x\in\R^d, \quad \mathcal{L}_{f_i} V(x):=\D V(x)\cdot f_i(x) \leq0.$$
		 
In this context the space $\R^d$ is naturally endowed with the Euclidean, but not canonical, norm defined by $\|x\|_P=\sqrt{x^TPx}$, and $V(\Phi_i^t(x))=\| \Phi_i^t(x)\|_P^2$ is strictly decreasing w.r.t. $t$ for $i=1,\ldots,p$. Indeed the $f_i$'s being analytic and globally asymptotically stable the Lyapunov function (even only weak) is strictly decreasing along all nonzero solutions of each individual vector field. Notice that these assumptions do not imply the GUAS property for all inputs (see Example \ref{example} and \cite{BJ11,JN13}).

Let us state the main result of this section.

\begin{theorem}\label{convrate}
Let $\{f_1,\ldots,f_p\}$ be a finite set of globally asymptotically stable, analytic and homogeneous vector fields assumed to share the common weak Lyapunov function $V(x)=x^TPx$.

Let $\delta$ be a positive number.\\
Then for all inputs with dwell-time $\delta$ the switched system under consideration is exponentially stable with a rate of convergence equal to $\displaystyle \beta(r,t)=r\min\{1,e^{-\frac{\ln(M)}{2}}e^{{\frac{\ln(M)}{2\delta}}t}\}$ where $M=M(\delta)<1$ is the following decreasing function of $\delta$:
$$
M(\delta)=\max_{i=1,\dots ,p}\{\max\{\norm{\Phi_i^\delta(x)}_P;\ \norm{x}_P=1\}\}.
$$
\end{theorem}

\begin{proof}
Let us denote by $\mathcal{S}$ the unit sphere for the norm $\norm{\cdot}_P$. Since we are dealing with homogeneous systems we can assume that the initial condition $x$ belongs to $\mathcal{S}$.

For each $i\in\{1,\ldots,p\}$, let $m_i=\max_{x\in\mathcal{S}} \norm{\Phi_i^\delta(x)}_P$. Since the norm is strictly decreasing along the nonzero solutions of each $f_i$ we have for a fixed $x\in\mathcal{S}$, $\norm{\Phi_i^\delta(x)}_P<1$. By compacity of the unit sphere $\mathcal{S}$ and continuity with respect to the initial condition we get $0<m_i<1$. 

Of course $M=\max_{1\leq i\leq p} m_i$ also satisfies $0<M<1$. 

Let us now consider an input $u$ with dwell-time $\delta$. Let $t$ be a positive number and $0=t_0<t_1<\ldots<t_k<t_{k+1}=t$ be the switching times in the interval $[0,t]$.

For each $j=0,\ldots,k$ we have 
	$$t_{j+1}-t_j=n_j\delta+r_j,$$
where $0\leq r_j<\delta$, $n_k\geq 0$, and $n_j\geq 1$ for $j=0,\dots,k-1$.
Let $\displaystyle n=\sum_{j=0}^k n_j\geq k$ be the number of disjoint intervals of length $\delta$ of $[0,t]$ on which the input $u$ is constant. Summing the previous equalities, we get
	$$t=n\delta+\sum_{i=0}^k r_j<n\delta+(k+1)\delta\leq (2n+1)\delta .$$
That equality can be rewritten:
	$$n\delta>\dfrac{t-\delta}{2},$$	
and we can conclude by homogeneity, and because $\displaystyle \ln(M)<0$, that
\begin{align}
\displaystyle 	\Vert \Phi_u^t(x) \Vert_P &\leq M^n= e^{n\delta \ln(M)/\delta}	\notag \\
\displaystyle &< e^{\frac{t-\delta}{2}\frac{\ln(M)}{\delta}}
\displaystyle =e^{-\frac{\ln(M)}{2}}e^{{\frac{\ln(M)}{2\delta}}t}.	\notag
\end{align}
To finish the norm is nonincreasing along the trajectories, and the convergence rate is consequently:
$$
\beta(r,t)=r\min\{1,e^{-\frac{\ln(M)}{2}}e^{{\frac{\ln(M)}{2\delta}}t}\}.
$$
\end{proof}

\begin{remarks}
\begin{enumerate}
    \item In the previous  proof, if $t$ is a switching time, then we get $n\geq k+1$ and finally $\Vert \Phi_u^t(x) \Vert_P\leq e^{{\frac{\ln(M)}{2\delta}}t}$. Actually the "overshot" $e^{-\frac{\ln(M)}{2}}$ is due to the fact that the decreasing rate is unknown between the last switching time $t_k$ and $t=t_{k+1}$ whenever $t_{k+1}-t_k<\delta$.
	\item If $x\in\R^d\setminus\{0\}$ is such that $\mathcal{L}_{f_i} V(x)=0$ for some $i$ then $\norm{\Phi_i^t(x)}_P=\norm{x}_P+\text{o}(t^k)$ where $k\geq1$ is an odd integer. In that case, the function $M=M(\delta)$ is at most polynomial.
	\item The weak Lyapunov function need not be quadratic. Indeed the properties of the norm used in the proof are analycity and homogeneity, and it is sufficient for the Lyapunov function $V$ to be the square of an analytic norm.
\end{enumerate}
\end{remarks}

\vskip 0.2cm

Next we give an example of a linear switched system which is asymptotically stable for \emph{dwell-time} inputs but not for all ones.

\begin{example}\label{example}
(see \cite{BM05})
Consider the switched system consisting in the two following Hurwitz matrices
	$$B_1=\begin{pmatrix}  0 & -1 \\ 1 & -1\end{pmatrix}$$
and
	$$B_2=\begin{pmatrix} 0 & 1 \\ -1 & -1\end{pmatrix}.$$	
They admit the identity matrix $I$ as a common weak quadratic Lyapunov function. \\
The switched system is not GUAS since the convex combination $\frac{1}{2}(B_1+B_2)$ is not Hurwitz.
However it can be shown (see \cite{BM05,JN13}) that the switched system is asymptotically stable for dwell-time inputs.\\
The following figure shows the graph of the function $M=M(\delta)$. 

\includegraphics[width=13cm]{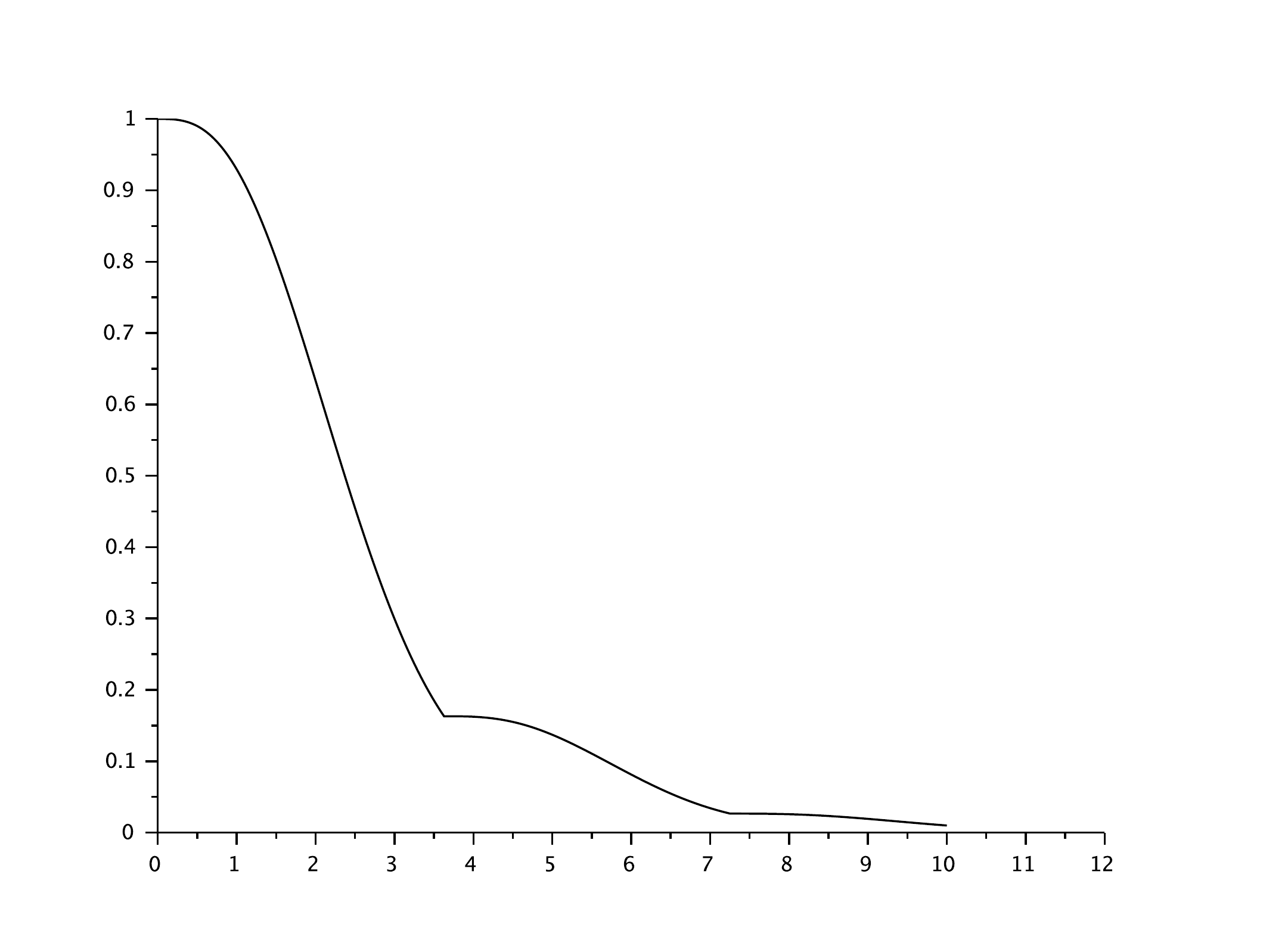}

In the next figure, the graph of the function $\beta(1,\cdot)$ is plotted for different values of the paramater $\delta$. 

\includegraphics[width=13cm]{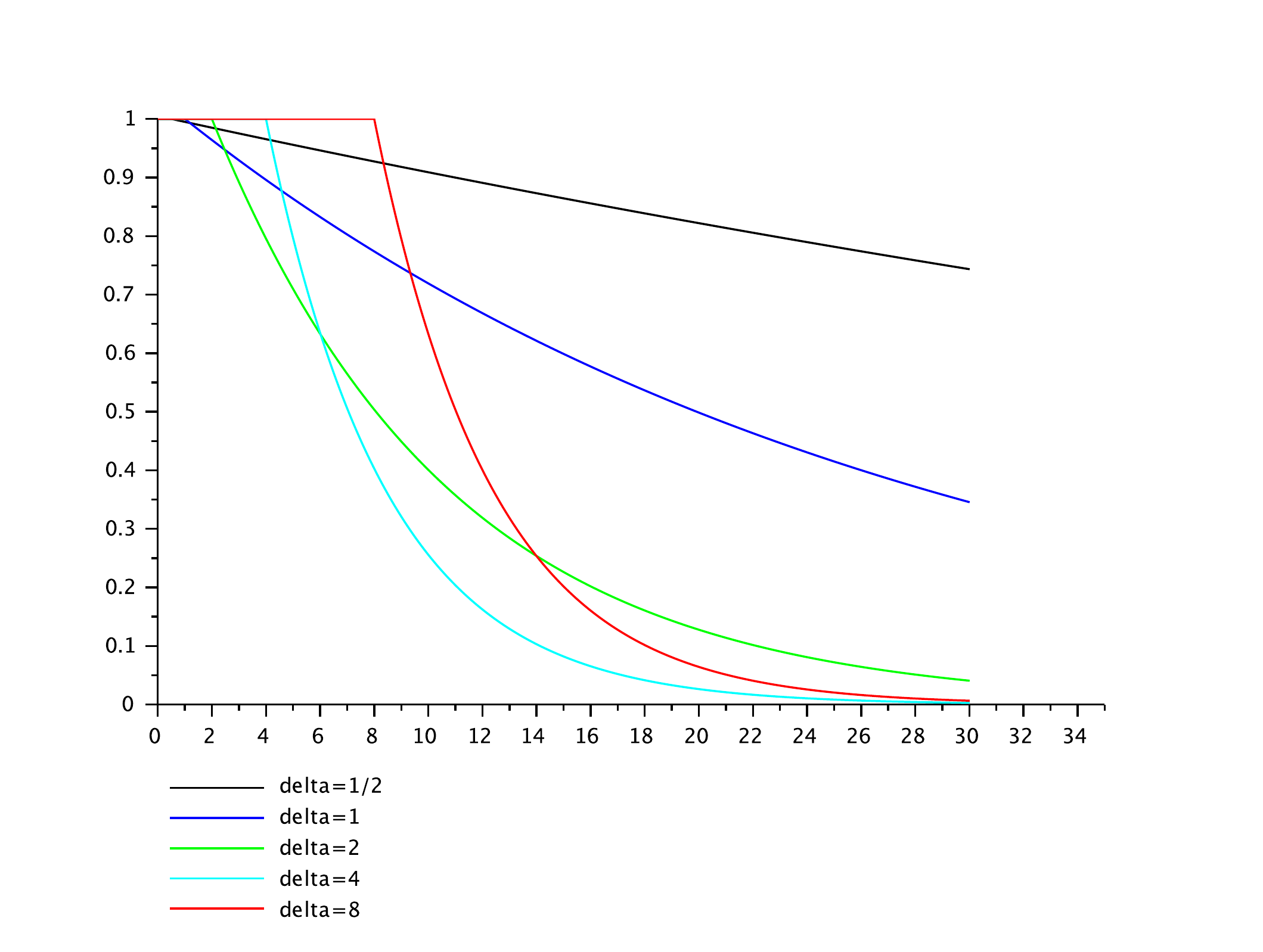}

\end{example}

\vskip 0.2cm


\section{Extension to general nonlinear switched systems}\label{SectionGeneral}

The aim of this section is to extend the previous result to general nonlinear switched systems: we no longer assume the vector fields  $\{f_1,\ldots,f_p\}$ to be homogeneous and the weak Lyapunov function $V$ to be quadratic. More accurately we consider the following setting:
\begin{enumerate}
\item[(i)] the vector fields $\{f_1,\ldots,f_p\}$ are analytic and globally asymptotically stable at the origin;
\item[(ii)] they share the common analytic weak Lyapunov function $V$.
\end{enumerate}
However we are led to add two other assumptions:
\begin{enumerate}
\item[(iii)] The linearization $B_i$ of the vector field $f_i$ at the origin is Hurwitz, for $i=1,\dots, p$. This assumption cannot be avoided. Actually if a central manifold at the origin would exist for one of the vector field, then the convergence for this vector field could be arbitrarily slow.
\item[(iv)] The  Lyapunov function $V$ is proper and its Hessian $H$ at the origin is positive definite. Indeed we need the level sets $\{V\leq r\}$ to be compact, and also to define a norm related to $V$.
\end{enumerate}
We therefore choose the following Euclidean norm on $\R^d$:
$$
\norm{x}_H=\sqrt{\frac{1}{2}x^THx}.
$$
Since the differential of $V$ vanishes at the origin, we get:
$$
V(x)=\frac{1}{2}x^THx+\ \mbox{higher order terms}=\norm{x}_H^2+O(\norm{x}_H^3).
$$

\begin{lemma}\label{normapproximation}
There exists a positive constant $\rho$ such that:
$$
\forall x\in\mathbb{R}^d\ \mbox{ s.t. } \norm{x}_H\leq \rho,\qquad \frac{1}{2}\norm{x}_H^2\leq V(x) \leq 2 \norm{x}_H^2.
$$
\end{lemma}
\begin{proof}
Straightforward from $V(x)=\norm{x}_H^2+O(\norm{x}_H^3)$.
\end{proof}

\begin{lemma}\label{approximation}
The quadratic function $x\longmapsto\norm{x}_H^2$ is a weak Lyapunov function for the $B_i$'s, $i=1,\dots,p$.
\end{lemma}
\begin{proof}
For each $f_i$ one has $f_i(x)=B_ix+O(\norm{x}_H^2)$. Since $V$ is a weak Lyapunov function for $f_i$, we get:
$$
\forall x\in\mathbb{R}^d \quad 0\geq \mathcal{L}_{f_i}V(x)=2\langle x,B_ix\rangle_H +O(\norm{x}_H^3),
$$
where $\langle\ ,\ \rangle_H$ stands for the scalar product associated to the Euclidean norm $\norm{.}_H$. Assume that $\langle x,B_ix\rangle_H>0$ for some $x$, then we would have $\mathcal{L}_{f_i}V(\lambda x)=2\langle\lambda x,B_i(\lambda x)\rangle_H +O(\norm{\lambda x}_H^3)>0$ for $\lambda>0$ small enough, a contradiction. Consequently we get the desired inequality:
$$
\forall x\in \mathbb{R}^d \qquad \langle x,B_ix\rangle_H\leq 0.
$$
\end{proof}

\begin{theorem}
Let $\delta$ and $R$ be two positive real numbers. Under the assumptions $(i)$ to $(iv)$ the switched system under consideration is exponentially stable for all inputs with dwell-time $\delta$ and all initial conditions $x\in \{V\leq R\}$: there exist two positive constants $\alpha$ and $\gamma$ (computed in the proof) such that
$$
\forall x\in \{V\leq R\} , \ \ \forall t\geq 0,\qquad V(\Phi_u^t(x))\leq \min\{1,\alpha e^{-\gamma t}\}V(x).
$$
\end{theorem}

\begin{proof}
Let $\delta>0$. We will first analyze the behaviour of the system in a neighborhood of the origin, and then in an annulus $r<V(x)<R$.

According to Lemma \ref{approximation}, the square $\norm{.}_H^2$ of the norm is a common quadratic weak Lyapunov function for the linear vector fields $x\longmapsto B_ix$ which are GAS by assumption. Consequently there exists a positive constant $m<1$ such that
$$
\forall x\in \mathbb{R}^d, \ \ \forall i=1,\dots,p \qquad \norm{e^{\delta B_i}x}_H\leq m\norm{x}_H.
$$
Since $B_i$ is the linearization of the vector field $f_i$ at the origin, we have
	$$\Phi_i^\delta(x)=e^{\delta B_i}x+\text{o}(\norm{x}_H),$$
and consequently
	$$\displaystyle \dfrac{\norm{\Phi_i^\delta(x)}_H}{\norm{x}_H} =\dfrac{\norm{e^{\delta B_i}x}_H}{\norm{x}_H}+\dfrac{\text{o}(\norm{x}_H)}{\norm{x}_H}\leq m+\dfrac{\text{o}(\norm{x}_H)}{\norm{x}_H} \longrightarrow_{\norm{x}_H\to 0} m<1.$$
In other words there exists positive constants $m\leq m_1<1$ and $r_1$ such that
$$
\forall x\in \mathbb{R}^d \mbox{ with } \norm{x}_H\leq r_1, \ \ \forall i=1,\dots,p \qquad \norm{\Phi_i^\delta(x)}_H\leq m_1\norm{x}_H.
$$	
With the same computation as in the previous section we get for all inputs $u$ with dwell-time $\delta$:
$$
\forall x\in \mathbb{R}^d \mbox{ with } \norm{x}_H\leq r_1, \  \forall t\geq 0 \qquad \norm{\Phi_u^t(x)}_H\leq e^{-\frac{\ln(m_1)}{2}}e^{{\frac{\ln(m_1)}{2\delta}}t}\norm{x}_H.
$$	
We can assume $r_1\leq \rho$, where $\rho$ is the positive constant defined in Lemma \ref{normapproximation}. Thanks to that lemma the previous inequality entails:
$$
\forall x\in \mathbb{R}^d \mbox{ with } \norm{x}_H\leq r_1, \  \forall t\geq 0 \qquad V(\Phi_u^t(x))\leq 4e^{-\ln(m_1)}e^{{\frac{\ln(m_1)}{\delta}}t}V(x).
$$
(Notice that it is the square of the norm which is involved in the Lemma).

As $V$ is non increasing along the trajectories, we actually obtain $\displaystyle V(\Phi_u^\delta(x))\leq \min\{1, 4e^{-\ln(m_1)}e^{{\frac{\ln(m_1)}{\delta}}t}\}V(x)$.

Now let $r$ be a positive constant such that $\{V< r\}$ be included in the open ball $\{\norm{x}_H<r_1\}$.

Since the annulus $A(r,R)=\{r\leq V\leq R\}$ is compact and $V$ strictly decreasing along the trajectories, there exists a positive constant $m_2<1$ such that:
$$
\forall x\in A(r,R), \ \ \forall i=1,\dots,p \qquad V(\Phi_i^\delta(x))\leq m_2V(x).
$$	
With a computation similar to the one of the previous section we get for all inputs $u$ with dwell-time $\delta$:
$$
\forall x\in A(r,R), \qquad V(\Phi_u^t(x))\leq \min\{1,e^{-\frac{\ln(m_2)}{2}}e^{{\frac{\ln(m_2)}{2\delta}}t}\}V(x)
$$ 
\textbf{as long as $\Phi_u^t(x)$ belongs to $A(r,R)$}.

Let $\alpha=4e^{-\ln(m_1)}e^{-\frac{\ln(m_2)}{2}}$ and $\gamma=\min\{-\frac{\ln(m_1)}{\delta},-\frac{\ln(m_2)}{2\delta}\}$, then we obtain for all inputs $u$ with dwell-time $\delta$
$$
\forall x\in \{V\leq R\} , \ \ \forall t\geq 0,\qquad V(\Phi_u^t(x))\leq \min\{1,\alpha e^{-\gamma t}\}V(x).
$$
In order to state this inequality for a point $x$ such that $\norm{x}_H>r_1$ and for a time $t$ such that $V(\Phi_u^t(x))<r$ we should consider an intermediate time $\tau$ such that $V(\Phi_u^{\tau}(x))\geq r$ and $\norm{\Phi_u^{\tau}(x))}_H\leq r_1$. The other cases are straightforward.

\end{proof}


\section{Conclusion}

In this work we have presented three results related to the convergence rate of switched systems. The first one states that if a (non linear) switched system is asymptotically stable for some class $\mathcal{U}$ of inputs but not GUAS then the solutions converge to the origin arbitrarily slowly.

For the second and third results we consider switched systems for which a common weak quadratic 
Lyapunov function exists, and we provide explicit convergence rates for inputs with a fixed dwell-time.

The aim of a future work will be to extend these last two results from inputs with dwell-time to nonchaotic ones.

\newpage
	

\bibliographystyle{plain} 
\bibliography{Bibliographie}

\end{document}